\documentclass[11pt]{article}
\usepackage{amsmath,amssymb, amscd}

\newtheorem{propo}{{\bf Proposition}}[section]
\newtheorem{coro}[propo]{{\bf Corollary}}
\newtheorem{lemma}[propo]{{\bf Lemma}} \newtheorem{theor}[propo]{{\bf
Theorem}} 

\newenvironment{proof}{{\bf Proof.}}{$\Box$}

\def\N{{\mathbb N}}

\begin{document}

\vspace*{1.0in}

\begin{center}  ON $n$-MAXIMAL SUBALGEBRAS OF LIE ALGEBRAS
\end{center}
\bigskip

\begin{center} DAVID A. TOWERS 
\end{center}
\bigskip
\centerline {Department of
Mathematics, Lancaster University} \centerline {Lancaster LA1 4YF,
England} \centerline {d.towers@lancaster.ac.uk}
\bigskip

\begin{abstract}
 A chain $S_0 < S_1 < \ldots < S_n = L$ is a {\em maximal chain} if each $S_i$ is a maximal subalgebra of $S_{i+1}$. The subalgebra $S_0$ in such a series is called an {\em $n$-maximal} subalgebra. There are many interesting results concerning the question of what certain intrinsic properties of the maximal subalgebras of a Lie algebra $L$ imply about the structure of $L$ itself. Here we consider whether similar results can be obtained by imposing conditions on the $n$-maximal subalgebras of $L$, where $n>1$.
\end{abstract}

\section{Introduction}
Throughout $L$ will denote a finite-dimensional Lie algebra over a field $F$. A chain $S_0 < S_1 < \ldots < S_n = L$ is a {\em maximal chain} if each $S_i$ is a maximal subalgebra of $S_{i+1}$. The subalgebra $S_0$ in such a series is called an {\em $n$-maximal} subalgebra. There are many interesting results concerning the question of what certain intrinsic properties of the maximal subalgebras of a Lie algebra $L$ imply about the structure of $L$ itself. For example: all maximal subalgebras are ideals of $L$ if and only if $L$ is nilpotent (see \cite{barnes}); all maximal subalgebras of $L$ are c-ideals of $L$ if and only if $L$ is solvable (see \cite{cideal}); if $L$ is solvable then all maximal subalgebras have codimension one in $L$ if and only if $L$ is supersolvable (see \cite{barnes1}); $L$ can be characterised when its maximal subalgebras satisfy certain lattice-theoretic conditions, such as modularity (see \cite{var}). Our purpose here is to consider whether similar results can be obtained by imposing conditions on the $n$-maximal subalgebras of $L$, where $n>1$.
\par

Similar studies have proved fruitful in group theory (see, for example, \cite{bel}, \cite{hup} and \cite{mann}). For Lie algebras the following result was established by Stitzinger.

\begin{theor} (Stitzinger, \cite[Theorem]{stit})\label{t:stit} Every $2$-maximal subalgebra of $L$ is an ideal of $L$ if and only if
one of the following holds:
\begin{itemize}
\item[(i)] $L$ is nilpotent and $\phi(L)= \phi(M)$ for all maximal subalgebras $M$ of $L$;
\item[(ii)] $\dim L = 2$; or
\item[(iii)] $L$ is simple and every proper subalgebra is one-dimensional.
\end{itemize}
\end{theor}
\bigskip

In the above result $\phi(L)$ denotes the Frattini ideal of $L$; that is, the largest ideal contained in the intersection of the maximal subalgebras of $L$. Our first objective in the next section is to find a similar characterisation of Lie algebras in which all $2$-maximal subalgebras are subideals, and then those in which they are nilpotent. In section three we consider when all $3$-maximals are ideals, and when they are subideals. In the final section we look at the situation where every $n$-maximal subalgebra is a subideal.
\section{$2$-maximal subalgebras}
First, the following observations will be useful.

\begin{lemma}\label{l:fi} Let $A/B$ be a chief factor of $L$ with $A \subseteq \phi(L)$. Then $A/B$ is an irreducible $L/\phi(L)$-module.
\end{lemma}
\begin{proof} The nilradical, $N$, of $L$ is the intersection of the centralizers of the factors in a chief series of $L$, by \cite[Lemma 4.3]{bg}. Since $\phi(L) \subseteq N$ this implies that $[A,\phi(L)] \subseteq B$ and so the multiplication of $L$ on $A$ induces a module action of $L/\phi(L)$ on $A/B$. Hence $A/B$ can be viewed as an irreducible $L/\phi(L)$-module.
\end{proof}
\bigskip

We will refer to a chief factor such as is described in Lemma \ref{l:fi} as being {\em below} $\phi(L)$.

\begin{lemma}\label{l:nmax} If every $n$-maximal subalgebra of $L$ is a subideal of $L$, then every $(n-1)$-maximal subalgebra is nilpotent.
\end{lemma}
\begin{proof} Let $J$ be an $(n-1)$-maximal subalgebra of $L$. Then every maximal subalgebra $I$ of $J$ is an $n$-maximal subalgebra of $L$ and so is a subideal of $L$, and thus of $J$. It follows that $I$ is an ideal of $J$, and hence that $J$ is nilpotent, by \cite{barnes}. 
\end{proof}

\begin{theor}\label{t:sub}  Every $2$-maximal subalgebra of $L$ is a subideal of $L$ if and only if
one of the following holds:
\begin{itemize}
\item[(i)] $L$ is nilpotent;
\item[(ii)]  $L=N+Fx$ where $N$ is the nilradical, $N^2=0$ and ad\,$x$ acts irreducibly on $N$; or
\item[(iii)] $L$ is simple with every proper subalgebra one dimensional.
\end{itemize}
\end{theor}
\begin{proof} Let every $2$-maximal of $L$ be a subideal of $L$. If $L$ is simple then (iii) holds with $\phi(L)=0$. So suppose that $N$ is a maximal ideal of $L$. Since $N$ will be contained in a maximal subalgebra of $L$ it will be nilpotent, by Lemma \ref{l:nmax}.
\par

 Suppose first that $N \not \subseteq \phi(L)$. Then there is a maximal subalgebra $M$ of $L$ such that $L=N+M$. Since $M$ is nilpotent, $L$ is solvable. Moreover, $L$ is nilpotent or minimal non-nilpotent. Suppose that $L$ is minimal non- nilpotent and $\phi(L) \neq 0$, so $L=N+Fx$ where $N$ is the nilradical of $L$, $N^2=\phi(L)$ and ad\,$x$ acts irreducibly on $N/N^2$, by \cite[Theorem 2.1]{nilp}. But now $\phi(L)+Fx$ is a maximal subalgebra of $L$ and any $2$-maximal subalgebra of $L$ containing $Fx$ would have to be contained in a proper ideal of $L$, which would be nilpotent, by Lemma \ref{l:nmax}, and so contained in $N$. It follows that $\phi(L)=0$. Hence either (i) or (ii) holds.
\par

So suppose now that $N \subseteq \phi(L)$. Then $N=\phi(L)$ and $L/\phi(L)$ is simple with every proper subalgebra one dimensional. Now $N+Fs$ is a maximal subalgebra of $L$ for every $s \in S$, and, as in the preceding paragraph, any $2$-maximal subalgebra containing $Fs$ would be contained in $N$. It follows that $N=0$.
\par

Conversely, let $L$ satisfy (i), (ii) or (iii). If $L$ is nilpotent then every subalgebra of $L$ is a subideal of $L$. If (ii) holds then the maximal subalgebras of $L$ are $N$ and $Fx$, and so the $2$-maximal subalgebras are inside $N$ and so are subideals of $L$. If (iii) holds then the only $2$-maximal subalgebra is the trivial subalgebra.
\end{proof}
\bigskip

Note that, over a perfect field $F$ of characteristic zero or $p>3$, for $L$ to satisfy condition (iii) in Theorem \ref{t:stit}, it must be three-dimensional and $\sqrt{F} \not \subseteq F$, by \cite[Theorem 3.4]{chief}.
\par

Next we consider when all of the $2$-maximal subalgebras are nilpotent. We consider the non-solvable and solvable cases separately, as for the former case we require restrictions on the field $F$.

\begin{theor}\label{t:nonsolv}  Let $L$ be a non-solvable Lie algebra over an algebraically closed field $F$ of characteristic different from $2,3$. Then every $2$-maximal subalgebra of $L$ is nilpotent if and only if $L/\phi(L) \cong sl(2)$ and $sl(2)$ acts nilpotently on $\phi(L)$. If $F$ has characteristic zero or if $L$ is restricted, then $\phi(L)=0$. 
\end{theor}
\begin{proof} Suppose that  every $2$-maximal subalgebra of $L$ is nilpotent, and let $M$ be a maximal subalgebra of $L$. If $M$ is not nilpotent then there is an element $x \in M$ such that ad\,$x|_M$ has a non-zero eigenvalue, $\lambda$ say. But now $M=Fx+Fy$ since this is not nilpotent. Hence, every maximal subalgebra of $L$ is nilpotent or two dimensional; in particular, they are all solvable. If $F$ has characteristic $p>3$, it follows from \cite[Proposition 2.1]{varea} that $L/\phi(L) \cong sl(2)$. Moreover, all maximal subalgebras of $sl(2)$ are two dimensional, so $\phi(L)+Fx$ is nilpotent for every $x \in sl(2)$. The claim for characteristic zero is well known; that for the case when $L$ is restricted is \cite[Corollary 2.13]{varea}
\par

The converse is easy.
\end{proof}

\begin{theor}\label{t:solv} Let $L$ be a solvable Lie algebra over a field $F$. Denote the image of a subalgebra $S$ of $L$ under the canonical homomorphism onto $L/\phi(L)$ by $\bar{S}$. Then all $2$-maximal subalgebras of $L$ are nilpotent if and only if one of the following occurs:
\begin{itemize}
\item[(i)] $L$ is nilpotent;
\item[(ii)] $L$ is minimal non-nilpotent, and so is as described in \cite{nilp};
\item[(iii)] $\bar{L}=\bar{A} \dot{+} F\bar{b}$, where $\bar{A}$ is the unique minimal abelian ideal of $\bar{L}$ and $\phi(L)+Fb$ is minimal non-nilpotent; 
\item[(iv)] $\bar{L}= \bar{A} \dot{+} (F \bar{b_1} + F \bar {b_2})$, where $\bar{A}$ is a minimal abelian ideal of $\bar{L}$, $\bar{B}= F\bar{b_1}+ F\bar{b_2}$ is a subalgebra of $\bar{L}$ and $L/\phi(L)$ acts nilpotently on $\phi(L)$; or
\item[(v)] $\bar{L} = (\bar{A_1} \oplus \bar{A_2}) \dot{+} F\bar{b}$, where $\bar{A_1}$ and $\bar{A_2}$ are minimal abelian ideals of $\bar{L}$ and $L/\phi(L)$ acts nilpotently on $\phi(L)$.
\end{itemize}
\end{theor}
\begin{proof} Suppose that all $2$-maximal subalgebras of $L$ are nilpotent. Then $\bar{L} = (\bar{A_1} \oplus \ldots \oplus \bar{A_n}) \dot{+} \bar{B}$, where $\bar{A_i}$ is a minimal abelian ideal of $\bar{L}$ for each $i=1, \ldots, n$, $\bar{A_1} \oplus \ldots \oplus \bar{A_n}$ is the nilradical, $\bar{N}$, of $\bar{L}$ and $\bar{B}$ is a subalgebra of $\bar{L}$, by \cite[Theorem 7.3]{frat}. If $n>2$ we have $\bar{A_i}+\bar{B}$ is nilpotent for each $i=1, \ldots, n$. But then $\bar{L}$, and hence $L$, is nilpotent, by \cite[Theorem 6.1]{frat}. Suppose that $\dim \bar{B}>2$ and let $\bar{C}$ be a minimal ideal of $\bar{B}$. If $\dim \bar{C}=1$ then $\bar{N}+\bar{C}$ is a nilpotent ideal of $\bar{L}$, contradicting the fact that $\bar{N}$ is the nilradical of $\bar{L}$; if $\dim \bar{C}>1$ we have that $\bar{N}+F \bar{c}$ is nilpotent for each $\bar{c} \in \bar{C}$ which again implies that $\bar{N}+\bar{C}$ is a nilpotent ideal of $\bar{L}$. Finally, if $n=2$ and $\dim \bar{B}=2$ a similar argument produces a contradiction.
\par

So suppose next that $n=1$ and $\dim \bar{B}=1$. Then the maximal subalgebras of $L$ are $A$ and $\phi(L)+Fx$, where $x \notin N$. If $\phi(L)+Fb$ is nilpotent we have case (ii); if it is minimal non-nilpotent we have case (iii).
\par

Next let $n=1$ and $\dim \bar{B}=2$. If $B= Fb_1+Fb_2$ then $A+Fb_1$ and $A+Fb_2$ are maximal subalgebras of $L$, and so $\phi(L)+Fb_1$ and $\phi(L)+Fb_2$ are $2$-maximal subalgebras.  It follows that $B$ acts nilpotently on $\phi(L)$ and we have case (iv).
\par

Finally, suppose that $n=2$ and $\dim \bar{B}=1$. Maximal subalgebras are $A_1 \oplus A_2$, $A_1+Fb$ and $A_2+Fb$, and $\phi(L)+Fb$ is a $2$-maximal subalgebra. It follows that $Fb$ acts nilpotently on $\phi(L)$ and we have case (v).
\par

The converse is straightforward.
\end{proof}
\bigskip

If $S$ is a subalgebra of $L$, the {\em centraliser} of $S$ in $L$ is $C_L(S)=\{x \in L \colon [x,S]=0\}$.

\begin{coro}\label{c:solv} With the notation of Theorem \ref{t:solv}, if $L$ is solvable and $F$ is algebraically closed, then  all $2$-maximal subalgebras of $L$ are nilpotent if and only if one of the following occurs:
\begin{itemize}
\item[(a)] $L$ is nilpotent;
\item[(b)] $\dim L \leq 3$;
\item[(c)] $F$ has characteristic $p$, $\bar{L}= \oplus_{i=0}^{p-1} F\bar{a_i} + F \bar{b_1} + F \bar {b_2}$, where $[\bar{a_i},\bar{b_1}] = a_{i+1}$, $[\bar{a_i},\bar{b_2}]= (\alpha +i) \bar{a_i}$ for $i=0, \ldots, p-1$ ($\alpha \in F$, suffices modulo $p$), $[\bar{b_1}, \bar{b_2}] =\bar{b_1}$ and $L/\phi(L)$ acts nilpotently on $\phi(L)$; or 
\item[(d)] $\bar{L}= F\bar{a_1} + F \bar{a_2} + F \bar {b}$, where $[\bar{b},\bar{a_1}] = \bar{a_1}$, $[\bar{b},\bar{a_2}]= \alpha \bar{a_2}$ ($\alpha \in F$), $[\bar{a_1}, \bar{a_2}] =0$ and $L/\phi(L)$ acts nilpotently on $\phi(L)$.
\end{itemize}
\end{coro}
\begin{proof} We consider in turn each of the cases given in Theorem \ref{t:solv}. Clearly case (i) gives (a), and if case (ii) holds then $\dim L=2$ (see \cite{nilp}), which is included in (b). If case (iii) holds, then $\bar{A}$ and $\phi(L)$ are both one dimensional, and so we have (b) again.
\par

Consider next case (iv). Suppose first that $\bar{B}$ is abelian. Then $\dim \bar{A}=1$, by \cite[Lemma 5.6]{sf}. But now $\dim \bar{L}/C_{\bar{L}}(F\bar{a})\leq1$ so $\dim C_{\bar{L}}(F\bar{a}) \geq 2$, contradicting the fact that $C_{\bar{L}}(F\bar{a})= F \bar{a}$. Thus $B$ cannot be abelian.
\par

If $\bar{B}$ is non-abelian, then $\bar{B}=F\bar{b_1}+F\bar{b_2}$ where $[\bar{b_1},\bar{b_2}]=\bar{b_1}$.  If $F$ has characteristic zero, then $\dim \bar{A}=1$, by Lie's Theorem. But now, as in the previous paragraph, $\dim C_{\bar{L}}(F\bar{a}) \geq 2$, yielding the same contradiction. Hence $F$ has characteristic $p>0$. Then this algebra has a unique $p$-map making it into a restricted Lie algebra: namely $\bar{b_1}^{[p]} = 0, \bar{b_2}^{[p]} = \bar{b_2}$ (see \cite{sf}); its irreducible modules are of dimension one or $p$, by \cite[Example 1, page 244]{sf}. Once again we can rule out the possibility that  $\dim \bar{A}=1$. So suppose now that $\dim \bar{A}=p$.  Let $\alpha$ be an eigenvalue for ad\,$\bar{b_2}|_{\bar{A}}$, so $[\bar{a},\bar{b_2}] = \alpha\,\bar{a}$ for some $\bar{a} \in \bar{A}$. Then $[\bar{a}(\hbox{ad}\,\bar{b_1})^i,\bar{b_2}] = (\alpha + i) \bar{a} (\hbox{ad}\,\bar{b_1})^i$ for every $i$, so putting $\bar{a_i} = \bar{a} (\hbox{ad}\,\bar{b_1})^i$ we see that $F\bar{a_0} + \dots + F\bar{a_{p-1}}$ is $\bar{B}$-stable and hence equal to $\bar{A}$. We then have the multiplication given in (e).
\par

Finally, consider case (v). Then $\bar{A_1}$ and $\bar{A_2}$ are one-dimensional. Moreover, if $L$ is not nilpotent, the $\bar{b}$ must act non-trivially on at least one of them, $\bar{A_1}=F\bar{a_1}$, say. This gives the multiplication described in (f). 
\end{proof}

\section{$3$-maximals subalgebras}
We consider first Lie algebras all of whose $3$-maximal subalgebras are ideals. We shall need the following lemma, which is an easy generalisation of \cite[Lemma 2]{stit}.
\begin{lemma}\label{l:three} Suppose that every $n$-maximal subalgebra of $L$ is an ideal of $L$. Then every $(n-1)$-maximal subalgebra of $L$ is nilpotent and is either an ideal or is one dimensional.
\end{lemma}
\begin{proof} Let $K$ be a $(n-1)$-maximal subalgebra of $L$. The fact that $K$ is nilpotent follows from Lemma \ref{l:nmax}. Suppose that $\dim K>1$. Then $K$ has at least two distinct maximal subalgebras $J_1$ and $J_2$, by  \cite[Lemma 1]{stit}. These are $n$-maximal subalgebras of $L$ and so are ideals of $L$. Moreover, $K=J_1+J_2$ and so is an ideal of $L$.
\end{proof}

\begin{theor}\label{t:three}  Let $L$ be a solvable Lie algebra over a field $F$. Then every $3$-maximal subalgebra of $L$ is an ideal of $L$ if and only if one of the following holds:
\begin{itemize}
\item[(i)] $L$ is nilpotent and $\phi(K)= \phi(M)$ for every $2$-maximal subalgebra $K$ of $L$ and every maximal subalgebra $M$ of $L$ containing it; or
\item[(ii)] $\dim L \leq 3$.
\end{itemize}
\end{theor}
\begin{proof}  Suppose that every $3$-maximal subalgebra of $L$ is an ideal of $L$. Then  Lemma \ref{l:three} shows that $L$ is given by Theorem \ref{t:solv}. We consider each of the cases in turn, and use the notation of that result. Suppose first that $L$ is nilpotent and let $J$ be a $3$-maximal subalgebra of $L$, $K$ be any $2$-maximal subalgebra of $L$ containing it, and $M$ be any maximal subalgebra of $L$ containing $K$. Then $J$ is an ideal of $L$ and $M/J$ is two dimensional. It follows that $M^2 \subseteq J$ and so $M^2 \subseteq \phi(K)=K^2 \subseteq M^2$. Hence $\phi(K)=M^2=\phi(M)$.
\par

Now suppose that $\bar{L}=\bar{A} \dot{+} \bar{B}$, where $\bar{A}$ is the unique minimal ideal of $\bar{L}$ and $\bar{B}$ is a subalgebra of $\bar{L}$ with $\dim \bar{B} \leq 2$. This covers cases (ii), (iii) and (iv) of Theorem \ref{t:solv}. If $\dim \bar{A}>2$ then there is a proper subalgebra $\bar{C}$ of $\bar{A}$ which is a $3$-maximal subalgebra of $\bar{L}$, and so an ideal of $\bar{L}$, contradicting the minimality of $\bar{A}$. If $\dim \bar{B}=2$ then $A+Fb$ is a maximal subalgebra of $L$ for each $\bar{0} \neq \bar{b} \in \bar{B}$. It follows that $\phi(L)+Fb$ is a $2$-maximal subalgebra of $L$. If this is an ideal of $L$ then $F \bar{b}$ is a minimal ideal of $\bar{L}$, contradicting the uniqueness of $\bar{A}$. It follows from Lemma \ref{l:three} that it has dimension one, and so $\phi(L)=0$. Similarly, $\dim \bar{A}=2$ yields that $\phi(L)=0$. Hence $\phi(L) \neq 0$ implies that $\dim(L/\phi(L) \leq 2$ and thus that $L$ is nilpotent. So suppose that $\phi(L)=0$ and $\dim L=4$. Then $A+Fb$ is a maximal subalgebra for every $b \in B$, and so $Fa$ is a $3$-maximal subalgebra, and hence an ideal, of $L$ for every $a \in A$, contradicting the minimality of $A$. Thus $\dim L \leq 3$.  
\par

So, finally, suppose that case (v) of Theorem \ref{t:solv} holds. We have that $\phi(L)=0$ as in the paragraph above. Also, if $\dim A_i >1$ ($i=1,2$) there is a proper subalgebra $C$ of $A_i$ which is a $3$-maximal subalgebra, and hence an ideal, of $L$.  It follows that $\dim A_i=1$ for $i=1,2$ and $\dim L = 3$.
\par

Conversely, suppose that (i) or (ii) hold. If (ii) holds then every $3$-maximal is $0$ and thus an ideal of $L$, so suppose that (i) holds. Let $J$ be a $3$-maximal subalgebra of $L$. Then $J$ is a maximal subalgebra of a $2$-maximal subalgebra $K$ of $L$ and $M^2= \phi(M)= \phi(K) \subseteq J$ for every maximal subalgebra $M$ containing $K$. It follows that $J$ is an ideal of $M$. But now $\dim L/J=3$ and there are two maximal subalgebras $M_1$ and $M_2$ of $L$ containing $J$ with $L=M_1+M_2$. Since $J$ is an ideal of $M_1$ and $M_2$, it is an ideal of $L$. 
\end{proof}

\begin{theor}\label{t:3nonsolv} Let $L$ be a non-solvable Lie algebra over a field $F$. Then every $3$-maximal subalgebra of $L$ is an ideal of $L$ if and only if one of the following holds:
\begin{itemize}
\item[(i)] $L$ is simple, all $2$-maximal subalgebras of $L$ are at most one dimensional and at least one of them has dimension one; 
\item[(ii)] $L/Z(L)$ is a simple algebra, all of whose maximal subalgebras are one dimensional, $Z(L)= \phi(L)$ and $\dim Z(L)=0$ or $1$;
\item[(iii)] $L=S \dot{+} Fx$ where $S$ is a simple ideal of $L$ and all maximal subalgebras of $S$ are one dimensional.
\end{itemize}
\end{theor}
\begin{proof} Suppose that every $3$-maximal subalgebra of $L$ is an ideal of $L$. Clearly, if $L$ is simple then every $2$-maximal subalgebra has dimension at most one, by Lemma \ref{l:three}, and so satisfies (i) or (ii). So let $N$ be a maximal ideal of $L$. Suppose first that $\dim L/N=1$, so $L=N \dot{+} Fx$, say. Clearly $N$ has more than one maximal subalgebra, since otherwise it is one dimensional and $L$ is solvable. If $N$ has a maximal subalgebra $K_1$ that is an ideal of $L$, then $N=K_1+K_2$, where $K_2$ is another maximal subalgebra of $L$, and both $K_1$ and $K_2$ are nilpotent. But then $N$, and hence $L$, is solvable. It follows from Lemma \ref{l:three} that every maximal subalgebra of $N$ is one dimensional. Let $I$ be a non-trivial ideal of $N$. Then $\dim N/C_N(I) \leq 1$. But this implies that $\dim N=2$ and $L$ is solvable again. It follows that $N$ is simple with all maximal subalgebras one dimensional. Hence, $L$ is as in case (iii).
\par

So suppose now that $L/N$ is simple. Then all $2$-maximal subalgebras of $L/N$ have dimension at most one. Suppose first that $L/N$ has a one-dimensional $2$-maximal subalgebra $A/N$. Then $\dim A=1$, by Lemma \ref{l:three}, and so $N=0$ and we have case (i) again. So suppose now that all maximal subalgebras of $L/N$ are one dimensional. Then $N$ is nilpotent and if $K$ has codimension one in $N$, $K$ is an ideal of $L$. Moreover, $K+Fs$ is a $2$-maximal subalgebra of $L$ for every $s \notin N$. It follows from Lemma \ref{l:three} that $K=0$. Hence $\dim N=1$. But now $\dim L/C_L(N) \leq 1$, which implies that $N=Z(L)$. If $Z(L)= \phi(L)$ we have case (ii). If $Z(L) \neq \phi(L)$ then we have a special case of (iii).
\par

The converse is straightforward.  
\end{proof}

\begin{coro}\label{c:3nonsolv}   Let $L$ be a non-solvable Lie algebra over an algebraically closed field $F$ of characteristic different from $2,3$. Then every $3$-maximal subalgebra of $L$ is an ideal of $L$ if and only if $L \cong sl(2)$. 
\end{coro}
\begin{proof} Suppose that every $3$-maximal subalgebra of $L$ is an ideal of $L$. Then every $2$-maximal subalgebra of $L$ is nilpotent, so $L/\phi(L) \cong sl(2)$, by Theorem \ref{t:nonsolv}. But $\phi(L)=0$ by Theorem \ref{t:3nonsolv}. The converse is clear.
\end{proof}
\bigskip

Next we give a characterisation of those Lie algebras in which every $3$-maximal subalgebra is a subideal.

\begin{theor}\label{t:3solv} Let $L$ be a solvable Lie algebra over a field $F$. Then every $3$-maximal subalgebra of $L$ is a subideal of $L$ if and only if one of the following occurs:
\begin{itemize}
\item[(i)] $L$ is nilpotent;
\item[(ii)] $L=N \dot{+} Fb$ where $N$ is the nilradical, dim\,$N^2=1$, ad\,$b$ acts irreducibly on $N/N^2$, and $N^2+Fb$ is abelian;
\item[(iii)] $\bar{L}=\bar{A} \dot{+} F\bar{b}$, where $\bar{A}$ is the unique minimal abelian ideal of $\bar{L}$, $\phi(L)^2=0$ and $\phi(L)$ is an irreducible $Fb$-module; 
\item[(iv)] $L= A \dot{+} (F b_1 + F b_2)$, where $A$ is a minimal abelian ideal of $L$, and $B= Fb_1+ Fb_2$ is a subalgebra of $L$; or
\item[(v)] $L = (A_1 \oplus A_2) \dot{+} Fb$, where $A_1$ and $A_2$ are minimal abelian ideals of $L$.
\end{itemize} 
\end{theor}
\begin{proof} Suppose that every $3$-maximal subalgebra of $L$ is a subideal of $L$. Then $L$ is as described in Theorem \ref{t:solv}. We consider each of the cases in turn. In case (i) every subalgebra of $L$ is a subideal of $L$. Suppose that case (ii) holds, so $L=N \dot{+} Fb$ where $N/N^2$ is a faithful irreducible $Fb$-module and $N^2+Fb$ is nilpotent. Let $C$ be an ideal of $N^2+Fb$ of codimension one in $N^2$. Then $C+Fb$ is a $2$-maximal subalgebra of $L$. Suppose that $C \neq 0$. Then $b \in D$ where $D \subset C+Fb$ is a $3$-maximal subalgebra of $L$. Since $D$ is a nilpotent subideal of $L$, there is a $k \in \N$ such that $N$(ad\,$D)^k=0$. Since $b \in D$ and $N/N^2$ is faithful, this is impossible. Hence $D=0$ and $\dim N^2=1$. 
\par

Suppose that (iii) holds. Then $\phi(L)/\phi(L)^2$ is a faithful irreducible $Fb$-module, and so $\phi(L)^2+Fb$ is a $2$-maximal subalgebra of $L$. If $\phi(L)^2 \neq 0$ then $b \in D$ where $D \subset \phi(L)^2+Fb$ is a $3$-maximal subalgebra of $L$. But this yields a contradiction as in the preceding paragraph. 
\par

Suppose next that (iv) holds. Then we can choose $b_1, b_2$ so that $[\bar{b_1},\bar{b_2}]= \lambda \bar{b_2}$ where $\lambda=0,1$. Then $[\bar{A},\bar{b_2}]$ is an ideal of $\bar{L}$ and so is equal to $\bar{A}$, since, otherwise, $\bar{b_2} \in C_{\bar{L}}(\bar{A})= \bar{A}$. Now $\phi(L)+Fb_2$ is a $2$-maximal subalgebra of $L$. Thus, if $\phi(L) \neq 0$ we have that $b_2$ belongs to a $3$-maximal subalgebra of $L$, giving a contradiction again.
\par

Finally, suppose that (v) holds. Then $\phi(L)+Fb$ is a $2$-maximal subalgebra of $L$ and we conclude that $\phi(L)=0$ as above.
\par

Conversely, if any of these cases are satisfied then every $3$-maximal subalgebra of $L$ is inside the nilradical of $L$, and hence is a subideal of $L$. 
\end{proof}

\begin{propo}\label{p:3nonsolv}   Let $L$ be a non-solvable Lie algebra over an algebraically closed field $F$ of characteristic different from $2,3$. Then every $3$-maximal subalgebra of $L$ is a subideal of $L$ if and only if $L/\phi(L) \cong sl(2)$. 
\end{propo}
\begin{proof} Suppose that every $3$-maximal subalgebra of $L$ is a subideal of $L$. Then every $2$-maximal subalgebra of $L$ is nilpotent, so $L/\phi(L) \cong sl(2)$, by Theorem \ref{t:nonsolv}. Conversely, if $L/\phi(L) \cong sl(2)$ then every $3$-maximal subalgebra of $L$ is contained in $\phi(L)$, which is nilpotent, and so they are all subideals of $L$.
\end{proof}

\section{n-maximal subalgebras}
The following result was proved by Schenkman in \cite{schenk} for fields of characteristic zero, and can be extended to cover a large number of cases in characteristic $p$ by using a result of Maksimenko from \cite{mak}.

\begin{lemma}\label{l:nil} Let $I$ be a nilpotent subideal of a Lie algebra $L$ over a field $F$. If $F$ has characteristic zero, or has characteristic $p$ and $L$ has no subideal with nilpotency class greater than or equal to $p-1$, then $I \subseteq N$, where $N$ is the nilradical of $L$.
\end{lemma}
\begin{proof} If $F$ has characteristic zero this is \cite[Lemma 4]{schenk}. For the characteristic $p$ case we follow Schenkman's proof. Let $I$ be a nilpotent subideal of $L$ and suppose that $I=I_0 < I_1 < \ldots < I_n=L$ is a chain of subalgebras of $L$ with $I_j$ an ideal of $I_{j+1}$ for $j=0, \ldots, n-1$. Let $N_j$ be the nilradical of $I_j$ and let $x_j \in I_j$. Then $I \subseteq N_1$, since $I$ is a nilpotent ideal of $I_1$. Also $[I_j,x_{j+1}] \subseteq I_j$, and so ad\,$x_{j+1}$ defines a derivation of $I_j$ for each $j=0, \ldots, n-1$. Moreover, $N_j$ is a subideal of $L$ and so has nilpotency class less than $p-1$. It follows from \cite[Corollary 1]{mak} that $[N_j,x_{j+1}] \subseteq N_j$, and hence that $N_j$ is an ideal of $I_{j+1}$. But then $N_j \subseteq N_{j+1}$, and $I \subseteq N_1 \subseteq N_2 \subseteq \ldots \subseteq N_n=N$, as claimed.
\end{proof}
\bigskip

We will refer to the characteristic $p$ condition in the above result as $F$ having characteristic {\em big enough}. 

\begin{lemma}\label{l:nsub} Let $L$ be a Lie algebra over a field $F$. Consider the following two conditions:
\begin{itemize}
\item[(i)] every $n$-maximal subalgebra of $L$ is contained in $N$; and
\item[(ii)] every $n$-maximal subalgebra of $L$ is a subideal of $L$. 
\end{itemize}
Then (i) implies (ii) and, if $F$ has characteristic zero or big enough, (ii) implies (i).
\end{lemma}
\begin{proof} (i) $\Rightarrow$ (ii):  It is clear that any subideal of $N$ is a subideal of $L$. 
\par

\noindent (ii) $\Rightarrow$ (i):  Let $I$ be an $n$-maximal subalgebra of $L$ and suppose that it is a subideal of $L$. Then, under the extra hypothesis, it is a nilpotent subideal of $L$, by Lemma \ref{l:nmax} and so is contained in $N$, by Lemma \ref{l:nil}.
\end{proof}
\bigskip

Clearly, if $L$ is solvable then a necessary condition for Lemma \ref{l:nsub} (i) to hold is that $\dim L/N \leq n$, since there is a chain of subalgebras of length $n$ from $N$ to $L$. However, this condition is not sufficient, in general, as is clear from previous results and the following.

\begin{theor} Let $L$ be a supersolvable Lie algebra over a field $F$ of characteristic zero or big enough. Then every $n$-maximal subalgebra of $L$ is a subideal of $L$ if and only if either
\begin{itemize}
\item[(i)] $L$ is nilpotent; or
\item[(ii)] $\dim L \leq n$.
\end{itemize}
\end{theor}
\begin{proof} Suppose that every $n$-maximal subalgebra of $L$ is a subideal of $L$,  but that $L$ is not nilpotent, and let $N$ be the nilradical of $L$. Let 
\[ 0=A_0 < A_1 < \ldots < A_k =N < \ldots < A_r=L
\] be a chief series for $L$ through $N$. Then each chief factor is one dimensional since $L$ is supersolvable and so $r= \dim L$. Let $x \in A_r \setminus A_{r-1}$. Then
\[ Fx < A_1+Fx < \ldots < A_{r-1}+Fx=L
\] is a maximal chain of subalgebras of $L$, and $Fx$ is an $(r-1)$-maximal subalgebra of $L$. If $r>n$ it follows that $x$ belongs to an $n$-maximal subalgebra of $L$. Since $x \not \in N$ this contradicts Lemma \ref{l:nsub}.  
\end{proof}

\end{document}